\documentclass[a4paper,11pt]{amsart}
\usepackage{amsmath,amsthm,amssymb, color}

\numberwithin{equation}{section}
\usepackage{hyperref}
\usepackage{enumitem}

\theoremstyle{plain}
\newtheorem{theorem}{Theorem}[section]
\newtheorem{proposition}[theorem]{Proposition}
\newtheorem{lemma}[theorem]{Lemma}

\newtheorem{corollary}[theorem]{Corollary}

\theoremstyle{definition}

\newtheorem{remark}[theorem]{Remark}

\newtheorem{example}[theorem]{Example}

\newcommand{\Z}{\mathbb{Z}}
\newcommand{\N}{\mathbb{N}}

\newcommand{\Q}{\mathbb{Q}}
\newcommand{\R}{\mathbb{R}}

\newcommand{\mathscr}{}

\newcommand{\cB}{\mathcal{B}}
\newcommand{\GW}{\mathrm{GW}}
\newcommand{\Ham}{\mathrm{Ham}}

\begin{document}
\title[Gromov Width of graph associahedra]{Gromov Width of symplectic toric manifolds associated with graphs}

\author[S. Choi]{Suyoung Choi}
\address{Department of Mathematics, Ajou University, 206 Worldcup-ro, Suwon 16499, South Korea}
\email{schoi@ajou.ac.kr}

\author[T. Hwang]{Taekgyu Hwang}
\address{Department of Mathematics, Ajou University, 206 Worldcup-ro, Suwon 16499, South Korea}
\email{hwangtaekkyu@gmail.com}

\thanks{The authors were supported by the National Research Foundation of Korea Grant funded by the Korean Government (NRF-2019R1A2C2010989).}

\date{\today}
\maketitle

\begin{abstract}
	We give an explicit formula for the Gromov width for a class of symplectic toric manifolds constructed from simple graphs.
    As a corollary, we show a version of non-squeezing theorem with respect to the inclusion of connected graphs.
\end{abstract}

\section{Introduction}
A $2n$-dimensional symplectic manifold $(M, \omega)$ is called \emph{toric} if it admits an effective Hamiltonian action of the $n$-dimensional torus~$T^n$.
It is known due to Delzant~\cite{Delzant1988} that symplectic toric manifolds are classified by their image under the moment map $\mu \colon M \to \mathrm{Lie}(T^n)^\ast \cong \R^n$.
The image turns out to be a convex polytope $P$ satisfying the following property: at each vertex~$p$ of~$P$, the primitive outward normal vectors of the facets containing~$p$ form a basis for $\Z^n$. A polytope satisfying this property is called a \emph{Delzant polytope}.

Feichtner--Sturmfels~\cite{Feichtner-Sturmfels2005}, and Postnikov~\cite{Postnikov2009} independently, introduced an interesting family of Delzant polytopes called nestohedra.
A \emph{building set} $\cB$ on a finite set $[n+1] := \{1, 2, \ldots, n+1\}$ is a collection of nonempty subsets of $[n+1]$ such that
\begin{itemize}
  \item $\cB$ contains all singletons $\{i\}$ for $i \in [n+1]$, and
  \item if $I, J \in \cB$ and $I \cap J \neq \emptyset$, then $I \cup J \in \cB$.
\end{itemize}
For $I \subset [n+1]$, let $\Delta_I$ be the simplex given by the convex hull of points~$e_i$, $i \in I$, where $e_i$ is the $i$th coordinate vector of $\R^{n+1}$.
Then the \emph{nestohedron}~$P_\cB$ is defined to be the Minkowski sum of simplices
\begin{equation}\label{eq:Minkowski_sum}
    P_\cB = \sum_{I \in \cB} \Delta_I \subset \R^{n+1}.
\end{equation}
In the case when $[n+1] \in \cB$, the polytope~$P_\cB$ has maximal dimension~$n$ and we regard~$P_\cB$ as a Delzant polytope in~$\mathbb{R}^n$ (see for example \cite[Proposition~7.10]{Postnikov2009}).

A building set can be constructed from a graph as follows (see \cite{CD2006, ToledanoLaredo2008}).
Let $G$ be a simple graph with the vertex set $[n+1]$. Then
\begin{equation}\label{eq:graph_building_set}
	\cB(G) := \{ I \subset [n+1] \mid \text{the subgraph $G|_I$ induced by $I$ is connected}\}
\end{equation}
is a building set on $[n+1]$, and the corresponding nestohedron $P_{\cB(G)}$ is called a \emph{graph associahedron}. The corresponding symplectic toric manifold is the main object of our study.

Symplectic geometry is very flexible in general. For instance, any smooth function on a symplectic manifold induces a family of symplectomorphisms by the flow of the Hamiltonian vector field. On the other hand, the rigidity is somewhat hidden and harder to detect. It was Gromov~\cite{Gromov1985} who first found a symplectic invariant which measures the ``width'', hence fundamentally different from the volume. His non-squeezing theorem asserts that the ball~$B^{2n}(r)$ of radius~$r$ can be symplectically embedded into the product~$B^2(R) \times \R^{2n-2}$ if and only if $r \leq R$. This motivates the following definition of the \emph{Gromov width} of a symplectic manifold $(M^{2n}, \omega)$:
\begin{equation}
    w_G (M, \omega) := \sup \{ \pi r^2 \mid \text{$B^{2n}(r)$ symplectically embeds into $M^{2n}$} \}.
\end{equation}

While the definition is simple, we do not have a general formula for the Gromov width even if we restrict our attention to symplectic toric manifolds. There are some cases where we can compute the exact value from the Delzant polytope. We list some of them in the following. The only toric manifold of dimension $2$ is $\mathbb{P}^1$ and the Gromov width is given by the area. In the case when the dimension is~$4$, there is an algorithm using Cremona transformations; see Karshon--Kessler~\cite[Section~6]{Karshon-Kessler2017}. We also have a formula when the Delzant polytope is combinatorially equivalent to the product of simplices~\cite{Hwang-Lee-Suh2019}. The computation is more tractable in the Fano case due to a result by Lu \cite[Theorem~1.2]{Lu2006}, but there are Fano examples for which we do not know the exact value; see for example \cite[Example~5.7]{Hwang-Lee-Suh2019}.

In this paper, we provide a formula for the symplectic manifold determined by a graph associahedron~$P_{\cB(G)}$ in terms of a graph theoretic invariant of~$G$. The polytope~$P_{\cB(G)}$ turns out to have parallel facets whose normal vector induces a semifree circle action. This property enables us to find an upper bound for the Gromov width given by the distance between parallel facets; see Proposition~\ref{prop:key}. Moreover, as we will see from Corollary~\ref{cor:stabilization}, this upper bound remains the same even if we take the product with Euclidean spaces. Combining Proposition~\ref{prop:key} with some graph theoretic computations (see Section~\ref{sec:graph}), we obtain the following theorem.

\begin{theorem}\label{thm:main}
    Let $G$ be a simple graph with the vertex set $[n+1]$. Let $(M_G, \omega_G)$ denote the symplectic toric manifold determined by the graph associahedron~$P_{\cB(G)}$.
    Then the Gromov width is given by
	\begin{equation}\label{eq:main}
		w_G (M_G, \omega_G) = \min \{ k_i>1 \mid i \in [n+1] \} - 1,
	\end{equation}
    where $k_i:= k_i(G)$ is the number of the connected induced subgraphs of $G$ containing the vertex~$i$. The minimum is taken over all vertices~$i$ with $k_i>1$; if $k_i = 1$ for all $i$, then $M_G$ is a point and both sides of~\eqref{eq:main} are zeros by convention.
\end{theorem}

The class of graph associahedra includes some important families of simple polytopes, such as permutohedra, associahedra (or Stasheff polytopes, which was first introduced in homotopy theory \cite{Stasheff1963}), cyclohedra (or Bott-Taubes polytopes) and stellohedra, corresponding to the complete graphs, the path graphs, the circle graphs, and the star graphs, respectively. We write down the explicit values of the Gromov width in such cases.
\begin{example}
    Let $K_{n+1}, P_{n+1}, C_{n+1}$ and $K_{1,n}$ be the complete graph, the path graph, the circle graph, and the star graph with $n+1$ vertices, respectively.
    Then
    \begin{alignat*}{2}
        & w_G (M_{K_{n+1}}, \omega_{K_{n+1}}) &&= 2^n - 1, \\
        & w_G (M_{P_{n+1}}, \omega_{P_{n+1}}) &&= n, \\
        & w_G (M_{C_{n+1}}, \omega_{C_{n+1}}) &&= \frac{n(n+1)}{2}, \quad \text{ and }\\
        & w_G (M_{K_{1,n}}, \omega_{K_{1,n}}) &&= 2^{n - 1}.
    \end{alignat*}
\end{example}

\begin{corollary}\label{cor:subgraph}
	Let $G$ be a connected simple graph. If $H$ is a subgraph of~$G$, we have $w_G(M_H, \omega_H) \leq w_G(M_G, \omega_G)$. In the case when $H$ has fewer vertices than~$G$, the inequality is strict.
\end{corollary}
\begin{proof}
    By the definition of~$k_i$, we have $k_i(H) \leq k_i(G)$ for each vertex~$i$ of~$H$.
    If there exists a vertex $j$ of $G$ not in $H$, we can find a vertex~$m$ of~$H$ connected to~$j$ by a path in $G$ such that no edge is in $H$.
    In this case, we have $k_m(H) < k_j(G)$.
\end{proof}

The Gromov width of closed symplectic manifolds is not very useful as an obstruction to symplectic embeddings because a topological embedding of a closed manifold into a connected manifold of the same dimension is necessarily a homeomorphism. Motivated by the works on stabilized symplectic embeddings of ellipsoids (see for example, \cite{Cristofaro-Hind2018, Crisofaro-Hind-McDuff2018, Hind-Kerman2014, McDuff2018, Siegel2019}), we can still think of embeddings after taking products with Euclidean spaces. As the Gromov width is unchanged under the stabilization by Corollary~\ref{cor:stabilization} (a precise argument will be given in Section~\ref{sec:graph}), we obtain  the following analog of the non-squeezing theorem for symplectic toric manifolds obtained from graphs.

\begin{corollary}\label{cor:embedding}
    Let $G$ be a connected simple graph.
    Assume that $H$ is a connected subgraph of $G$ and $|H| = |G|-k$ with $k>0$, where $|G|$ denotes the number of vertices of~$G$.
    Then an embedding
    \begin{equation}\label{eq:squeeze}
    M_G \times \R^{2m} \hookrightarrow M_H \times \R^{2k+2m}
    \end{equation}
    can never be symplectic for any $m \geq 0$.
\end{corollary}

\begin{remark}
	In general, the Gromov width is not preserved under the stabilization as we see from the result of Lalonde~\cite{Lalonde1994} asserting that
	\[
		w_G(\Sigma_g \times \mathbb{R}^2) = \infty \quad \text{for a Riemann surface $\Sigma_g$ with genus $g \geq 1$}.
	\]
	 Also, the authors do not know whether there exists a topological obstruction to the embedding~\eqref{eq:squeeze}.
\end{remark}

\section{Estimation of the Gromov width} \label{sec:gromov_width}
In this section, we collect some results on the estimation of the Gromov width that we need to prove Theorem~\ref{thm:main}. The key result is Lemma~\ref{lem:key} on the upper bound as the distance between two parallel facets. More details and references can be found in~\cite{Hwang-Lee-Suh2019} where a similar strategy is used.

\subsection{Lower bound}
Let $(M,\omega)$ be a symplectic toric manifold of dimension~$2n$.
We are given a moment map
$$
    \mu \colon M \to \mathfrak{t}^\ast \simeq \R^n
$$
where the weight lattice is identified with $\Z^n$ by the last map.
A lower bound for the Gromov width can be obtained by looking at the image $P$ of the moment map.
For $\rho >0$, let $\Diamond (\rho)$ be the convex hull of $n$ line segments $L_1, \ldots, L_n$ such that
\begin{itemize}
  \item $\cap_{i=1}^n L_i$ is a point.
  \item The primitive vectors parallel to $L_i$ form an integral basis.
  \item The affine length of $L_i$ is $\rho$, where the affine length is the ratio of the length compared to the parallel primitive vector.
\end{itemize}

\begin{proposition}[Latschev--McDuff--Schlenk~\cite{Latschev-McDuff-Schlenk2013}, Section~4.2, see also Mandini-Pabiniak~\cite{Mandini-Pabiniak2018}, Proposition~5] \label{prop:lower_bound}
    If $P$ contains $\Diamond(\rho)$, then the Gromov width of $(M, \omega)$ is at least $\rho$.
\end{proposition}
\subsection{Upper bound}
Let $\iota \colon B^{2n}(r) \hookrightarrow (M, \omega)$ be a symplectic embedding. Gromov proved that $\pi r^2$ is bounded above by the area of a pseudoholomorphic curve passing through the point $\iota(0)$. In the following, we state the Gromov's theorem in terms of the Gromov--Witten invariants. Readers are referred to McDuff--Salamon~\cite[Chapter~7]{McDuff-Salamon2012} for the definitions and properties.
Given an integral homology class $A \in H_2(M, \Z)$ and $k$-tuples of cohomology classes $\alpha_i \in H^\ast(M, \Q)$, the genus zero Gromov--Witten invariant
 \[
 		\GW_{A,k}^M (\alpha_1, \alpha_2, \dots, \alpha_k) \in \Q
 \]
counts the number of $J$-holomorphic stable maps of genus zero in~$M$ which represents the class~$A$ and intersects with the cycles~$\alpha_i$. Let $[pt] \in H^\ast(M, \Q)$ denote the point class. We will use the notation~$\omega(A)$ to denote the symplectic area $\int_A \omega$.
\begin{theorem}[Gromov~\cite{Gromov1985}]\label{thm:Gromov}
	Suppose that $\GW_{A,k}^M ([pt], \alpha_2, \dots, \alpha_k) \neq 0$ for some nonzero $A \in H_2(M, \Z)$ and $\alpha_i \in H^\ast(M, \Q)$. Then the Gromov width of $(M, \omega)$ is at most $\omega(A)$.
\end{theorem}

The main tool we use to find a Gromov--Witten invariant as in Theorem~\ref{thm:Gromov} is given by the work of McDuff and Tolman~\cite{McDuff-Tolman_2006} on the Seidel representation~\cite{Seidel1997}. The \emph{Seidel morphism} is a group homomorphism
\begin{equation}\label{eq:Seidel_morphism}
	S \colon \pi_1(\Ham(M, \omega)) \to QH^0(M, \Lambda)^{\times}
\end{equation}
from the fundamental group of the Hamiltonian diffeomorphism group to the multiplicative group of the degree zero invertible elements in the (small) quantum cohomology ring. (We will use quantum cohomology rather than quantum homology. The cohomology version displays the degree computation more clearly; see \cite[Section~2.2]{McDuff-Tolman_2006}.) McDuff and Tolman developed methods to compute the image of~$S$ when the loop is represented by a Hamiltonian circle action.

In order to describe the result by McDuff and Tolman, we briefly review the definition of the quantum cohomology ring following \cite[Chapter~11]{McDuff-Salamon2012}.
Consider the Novikov ring
\begin{equation*}\label{eq:Novikov}
	\Lambda:= \Lambda^{\mathrm{univ}}[q,q^{-1}]
\end{equation*}
where $q$ is a variable of degree~$2$, and
\[
    \Lambda^{\mathrm{univ}}:= \left\{\sum_{i \in \N}
	r_i t^{\kappa_i} \;\Big|\; r_i \in \Q, \; \kappa_i \in \R, \; \lim_{i \to \infty} \kappa_i = \infty \right\}
\]
with $\deg t =0$.
The \emph{quantum cohomology ring} with coefficients in $\Lambda$ is an abelian group
\begin{equation*}
	QH^\ast(M, \Lambda):= H^\ast(M, \Q) \otimes_{\Q} \Lambda
\end{equation*}
together with the product $\ast$ defined as follows.
Given $A \in H_2(M, \Z)$, we will use the notation $c_1(A):= \langle c_1(M), A \rangle \in \Z$ for the first Chern number throughout the paper.
Let $a \in H^i(M, \Q)$ and $b \in H^j(M, \Q)$.
Then
\begin{equation}\label{eq:qp}
	a\ast b := \sum_{A \in H_2(M, \Z)} (a \ast b)_A \otimes q^{c_1(A)} t^{\omega(A)}
\end{equation}
where $(a \ast b)_A \in H^{i+j-2c_1(A)}(M, \Q)$ is defined uniquely by the condition
\begin{equation}\label{eq:GW}
	\int_X (a \ast b)_A \cup c = \GW_{A,3}^M (a,b,c)
\end{equation}
for all $c \in H^\ast(M, \Q)$.
This product extends linearly on $\Lambda$ and is called the \emph{quantum product}.
The quantum cohomology ring is an associative ring with unity $1 = [M]$.

Now we are ready to state a theorem by McDuff and Tolman. Let $(M, \omega)$ be a closed symplectic manifold equipped with a Hamiltonian $S^1$-action. We have a corresponding moment map
\[
	H \colon M \to \mathbb{R}.
\]
There is a choice of~$H$ up to constant addition and we will take the normalized one in the sense that $\int_M H \omega^n = 0$. Let $F_{\max}$ (respectively, $F_{\min}$) denote the fixed submanifold which is maximal (respectively, minimal) with respect to the moment map~$H$. We write
\begin{equation*}
	H_{\max}:= H(F_{\max}) \quad \text{and} \quad H_{\min}:= H(F_{\min})
\end{equation*}
for brevity. Given a fixed submanifold~$F$, the sum of the weights for the tangential representation at $x \in F$ does not depend on the choice of~$x$ and will be denoted by~$m(F)$.
\begin{theorem}[McDuff--Tolman~\cite{McDuff-Tolman_2006}, Theorem~1.10 and Lemma~3.10]\label{thm:MT}
	Let $u \in \pi_1(\Ham(M, \omega))$ be represented by a circle action with the normalized moment map $H \colon M \to \R$.
    Then
	\begin{enumerate}
		\item \label{MT1}
			The image of the Seidel morphism~\eqref{eq:Seidel_morphism} is given by
		\[
			S(u) = \sum_{B \in H_2(M, \Z)} a_{u,B} \otimes q^{m(F_{\max}) + c_1(B)} t^{-H_{\max} + \omega(B)}
		\]
		for some cohomology class $a_{u,B} \in H^{-2m(F_{\max}) - 2c_1(B)}(M)$.
		\item \label{MT2} In the case when $\omega(B) < 0$ or $\omega(B) = 0$ with $B \neq 0$, the cohomology class $a_{u,B}$ is zero.
		\item \label{MT3} Let $[N]$ denote the cohomology class represented by an $S^1$-invariant submanifold $N$ of $M$.
        Then
		\[
			\int_M a_{u,B} \cup [N] = 0
		\]
		unless the homology class $B \in H_2(M, \mathbb{Z})$ can be represented by an $S^1$-invariant $J$-holomorphic stable map intersecting both the fixed submanifold~$F_{\max}$ and the submanifold $N$.
	\end{enumerate}
\end{theorem}

A circle action is said to be \emph{semifree} if it is free outside the fixed point set.
There is a simple characterization in terms of the weights for Hamiltonian actions.
We give a proof of the following fact for the reader's convenience.
\begin{lemma}\label{lem:semifree}
	A Hamiltonian $S^1$-action is semifree if and only if all the weights at fixed points are one of $\pm1, 0$.
\end{lemma}
\begin{proof}
	Choose an $S^1$-invariant almost complex structure~$J$ so that the Riemannian metric $g(-,-):= \omega(-, J-)$ is $S^1$-invariant. Then the gradient vector field of a moment map $H \colon M \to \mathbb{R}$ is given by $-J\underline{X}$, where $\underline{X}$ is the fundamental vector field of the action. Since $J$ is $S^1$-invariant, the flow~$\gamma_s$ of the gradient vector field $-J\underline{X}$ commutes with the $S^1$-action~$\phi_t$.
	
	Now suppose that there is a point $p \in M$ whose stabilizer is $\mathbb{Z}/k$ for some $k>1$. Since $\gamma_s$ commutes with $\phi_t$, the points $\gamma_s(p)$ have stabilizer~$\mathbb{Z}/k$ for all $s \in \mathbb{R}$. This implies that the limit point as $s \to \infty$ is a fixed point having a multiple of~$k$ as a weight. This shows that the action is semifree if all the weights are one of $0, \pm1$. The other direction is obvious.
\end{proof}

\begin{lemma}\label{lem}
	Suppose that the $S^1$-action representing~$u$ in Theorem~\ref{thm:MT} is semifree and the maximal fixed submanifold~$F_{\max}$ has real codimension two.
    Then, $a_{u,B} = 0$ for any $B$ satisfying $c_1(B) = 1$ and $\omega(B) < H_{\max} - H_{\min}$.
\end{lemma}
\begin{proof}
	Since $\dim F_{\max} = 2n-2$ and the action is effective, the weights at~$F_{\max}$ are $(0, \dots, 0, -1)$.
	By Theorem~\ref{thm:MT}~\eqref{MT1} and the assumption that $c_1(B)=1$, the degree of the cohomology class~$a_{u,B}$ is zero.
    Then $a_{u,B}=0$ if and only if $\int_M a_{u,B} \cup [pt] = 0$.

    Suppose $a_{u,B} \neq 0$ and apply Theorem~\ref{thm:MT}~\eqref{MT3} for a point~$N$ in the minimal fixed submanifold~$F_{\min}$.
    Since the action is semifree, the gradient sphere joining two fixed points $x$ and $y$ has symplectic area $|H(x) - H(y)|$ (see for example \cite[Lemma~3.9]{McDuff-Tolman_2006}).
    The holomorphic stable map in Theorem~\ref{thm:MT}~\eqref{MT3}, whose image is connected, consists of gradient spheres joining two fixed points and the holomorphic spheres contained entirely in the fixed point set.
    Since the symplectic area of each sphere is positive, the holomorphic stable map has symplectic area at least $H_{\max} - H_{\min}$.
    This contradicts the assumption that $\omega(B) < H_{\max} - H_{\min}$.
\end{proof}

\begin{lemma}\label{lem:key}
	Let $S^1$ act semifreely on a closed symplectic manifold $(M, \omega)$ with moment map $H \colon M \to \R$.
	Suppose that the extremal fixed submanifolds~$F_{\min}$ and $F_{\max}$ with respect to~$H$ are of real codimension two.
    Then $w_G(M, \omega) \leq H_{\max} - H_{\min}$.
\end{lemma}

\begin{proof}
	Since the weights at $F_{\max}$ are $(0, \dots, 0, -1)$, we have
	\[
		S(u) = \sum_{B_1 \in H_2(M, \Z)} a_{u, B_1} \otimes q^{-1 + c_1(B_1)} t^{-H_{\max} + \omega(B_1)}
	\]
	by Theorem~\ref{thm:MT}~\eqref{MT1}. If we consider the opposite circle action, the corresponding moment map is given by $-H$ and the minimal fixed submanifold~$F_{\min}$ becomes maximal.
    Hence,
	\[
		S(-u) = \sum_{B_2 \in H_2(M, \Z)} a_{-u, B_2} \otimes q^{-1 + c_1(B_2)} t^{H_{\min} + \omega(B_2)}.
	\]
	Since $S$ is a group homomorphism, we have
	\begin{align*}
		1 &= S(u - u) = S(u) * S(-u) \\
		&=  \sum_{B_1, B_2 \in H_2(M, \Z)} a_{u, B_1} * a_{-u, B_2} \otimes q^{-2 + c_1(B_1 + B_2)} \ t^{-(H_{\max} - H_{\min}) +  \omega(B_1 + B_2)} \\
		&= \sum_{A, B_1, B_2 \in H_2(M, \Z)} (a_{u, B_1} * a_{-u, B_2})_A \\
		&\qquad \qquad \qquad \otimes q^{-2 + c_1(A + B_1 + B_2)} \ t^{-(H_{\max} - H_{\min}) +  \omega(A + B_1 + B_2)}. \quad \text{(see \eqref{eq:qp})}
	\end{align*}
	By comparing the coefficient of~$1$, we see that there exist homology classes $A', B_1', B_2' \in H_2(M, \Z)$ such that
	\begin{align}\label{eq:upper_bound}
		\begin{cases}
		(a_{u, B_1'} * a_{-u, B_2'})_{A'} &\neq 0, \\
		c_1 (A' + B_1' + B_2') &= 2, \\
		\omega (A' + B_1' + B_2') &= H_{\max} - H_{\min}.
		\end{cases}
	\end{align}
	We claim that $A' \neq 0$.
    Recall from Theorem~\ref{thm:MT}~\eqref{MT1} that
	\begin{equation}\label{eq:deg_a}
		\deg (a_{u, B_1'}) = 2 - 2c_1(B_1') \quad \text{and} \quad \deg (a_{-u, B_2'}) = 2 - 2c_1(B_2').
	\end{equation}
	Suppose $A' = 0$.
    Then it follows from~\eqref{eq:upper_bound} that $c_1(B_1') = c_1(B_2') =1$.
    Moreover, the values $\omega(B_1')$ and $\omega(B_2')$ are less than $H_{\max} - H_{\min}$ by Theorem~\ref{thm:MT}~\eqref{MT2} and~\eqref{eq:upper_bound}. Then $a_{u, B_1'} = a_{-u,B_2'} = 0$ by Lemma~\ref{lem}, which contradicts the first condition of~\eqref{eq:upper_bound}.
	
	Note from~\eqref{eq:qp} and~\eqref{eq:deg_a} that
	\begin{equation}\label{eq:degree}
		\deg (a_{u, B_1'} * a_{-u, B_2'})_{A'} = 4 - 2c_1(A' + B_1' + B_2') = 0.
	\end{equation}
	By \eqref{eq:GW}, \eqref{eq:upper_bound} and \eqref{eq:degree}, we have
	 \[
	 	\mathrm{GW}_{A',3}^M(a_{u, B_1'}, a_{-u, B_2'}, [pt]) = \int_M (a_{u, B_1'} * a_{-u, B_2'})_{A'} \cup [pt] \neq 0.
	\]
	Note that $\omega(B_i') \geq 0$ by Theorem~\ref{thm:MT}~\eqref{MT2}.
    Now the lemma follows from Theorem~\ref{thm:Gromov} and \eqref{eq:upper_bound}.
\end{proof}

The upper bound obtained in Lemma~\ref{lem:key} is stabilized in the following sense.
\begin{corollary}\label{cor:stabilization}
	Let $M$ satisfy the assumption of Lemma~\ref{lem:key}. Then for any positive integer~$k$, the Gromov width of $M \times \mathbb{R}^{2k}$ with the product symplectic form is at most $H_{\max} - H_{\min}$.
\end{corollary}
\begin{proof}
	Let $\iota \colon B^{2n+2k}(r) \hookrightarrow M \times \R^{2k}$ be a symplectic embedding. Since the image of~$\iota$ is contained in a compact region, by taking a sufficiently large lattice $\Lambda \subset \R^{2k}$ and setting $V:= \R^{2k}/\Lambda \cong T^{2k}$, we have a symplectic embedding $\iota \colon B^{2n+2k}(r) \hookrightarrow M \times V$. The $S^1$-action on~$M$ extends to an action on $M \times V$ by acting trivially on~$V$. It is straightforward to check this action satisfies the assumption of Lemma~\ref{lem:key} and we have $\pi r^2 \leq H_{\max} - H_{\min}$.
\end{proof}

Hamiltonian $T^n$-actions have plenty of sub-circle actions induced from the choice of a primitive vector $u \in \mathbb{Z}^n$. The weights for the induced action are obtained from the weights for the $T^n$-action by pairing with~$u$ considered as an element in the Lie algebra of~$T^n$. Since the primitive edge vectors of the Delzant polytope~$P$ are weights for the $T^n$-action, Lemma~\ref{lem:semifree} together with Lemma~\ref{lem:key} implies the following proposition.
\begin{proposition}\label{prop:key}
	Let $(M, \omega)$ be a symplectic toric manifold whose moment polytope is $P \subset \R^n$.
	Suppose that there exists a primitive vector $u \in \Z^n$ satisfying the following two conditions.
	\begin{itemize}
    		\item $\langle u, \eta \rangle \in \{0, \pm1 \}$ for any primitive vector $\eta$ parallel to an edge of $P$.
	    	\item $P$ has supporting hyperplanes of the form $\{ x \in \R^n \mid \langle x, u \rangle \leq \lambda \}$ and $\{ x \in \R^n \mid \langle x, u \rangle \geq \mu \}$.
	\end{itemize}
	Then the Gromov width of $(M, \omega)$ is at most $\lambda - \mu$.
\end{proposition}

\section{Graph associahedra}\label{sec:graph}
The proof of Theorem~\ref{thm:main} proceeds as follows. 
We first assume that the graph~$G$ is connected so that $P:= P_{\cB(G)}$ is $n$-dimensional. 
The upper bound is obtained from the desciption of the edges of~$P$ (Lemma~\ref{lem:edge_property}) and Proposition~\ref{prop:key}. 
The lower bound is given by Lemma~\ref{lem:lower_bound} and Proposition~\ref{prop:lower_bound}. 
Finally, the proof for the general case follows from the observation that $P$ is the product of the graph associahedra of the connected components.
\subsection{Faces of nestohedra}
Let $\cB$ be a building set on $[n+1]$, which is not necessarily obtained from a graph.
It is known that the nestohedron~$P_\cB$ can be described as the intersection of the hyperplane
$$
    H_\cB = \left\{ (x_1, \ldots, x_{n+1} ) \in \R^{n+1} \middle\vert~ \sum_{i=1}^{n+1} x_{i} = |\cB| \right\}
$$
with the halfspaces
$$
    H_{I, \geq} = \left\{ (x_1, \ldots, x_{n+1} ) \in \R^{n+1} \middle\vert~ \sum_{i \in I} x_{i} \geq |\cB|_I | \right\}
$$
for all $I \neq [n+1]$ in~$\cB$, where $\cB|_I = \{ J \in \cB \mid J \subset I\}$.
If $[n+1] \in \cB$, then $P_\cB$ is $n$-dimensional and this representation is irredundant. That is, each halfspace~$H_{I, \geq}$ with $I \neq [n+1]$ defines a facet, denoted by $F_I$, of $P_\cB$.
It should be noted that $F_{I_1} \cap \cdots \cap F_{I_k} \neq \emptyset$ if and only if the following two conditions are satisfied:
\begin{enumerate}[label = (\alph*)]\label{d}
  \item \label{a}
  	for any $i,j$, $1 \leq i < j \leq k$, either $I_i \subset I_j$, or $I_j \subset I_i$, or $I_i \cap I_j = \emptyset$;
  \item \label{b}
  if the sets $I_{i_1}, \ldots, I_{i_p}$ are pairwise disjoint for some $p \geq 2$, then $I_{i_1} \cup \cdots \cup I_{i_p} \not\in \cB$.
\end{enumerate}
We refer the reader to Buchstaber--Panov~\cite[Section~1.5]{Buchstaber-Panov2015} for details.

The following lemma gives a description on the edges of~$P_{\cB}$.
\begin{lemma} \label{lem:edge_property}
    Let $\cB$ be a building set on $[n+1]$ such that $[n+1] \in \cB$.
    Let $\eta$ be a primitive vector parallel to an edge of $P_{\cB} \subset \R^{n+1}$.
    Then, $\eta = e_j - e_k$ for some $1\leq j, k \leq n+1$, where $e_i$ is the $i$th coordinate vector of $\R^{n+1}$.
\end{lemma}
\begin{proof}
	Assume that $\eta = \sum_{i=1}^{n+1} a_i e_i$ is parallel to an edge $e = F_{I_1} \cap \cdots \cap F_{I_{n-1}}$ for some $I_1, \dots, I_{n-1} \in \cB$.
    Since $e \subset H_{\cB}$, we have
    \begin{equation}\label{eqn:edge_in_hyperplane}
        \langle \eta, \sum_{i=1}^{n+1} e_i \rangle = \sum_{i=1}^{n+1} a_i = 0.
    \end{equation}
    The edge~$e$ is also contained in the boundary of $H_{I_t, \geq}$, so
    \begin{equation}\label{eqn:edge}
        \langle \eta, \sum_{i \in I_t} e_i \rangle = \sum_{i \in I_t} a_i = 0
    \end{equation}
    for all $t = 1, \ldots, n-1$.

	We relabel $I_1, \dots, I_{n-1}$ so that $I_p \subset I_q$ implies $p \leq q$.
    Set $J_t:= \cup_{s \leq t} I_s$ with the convention $J_0:= \emptyset$ and $J_n:=[n+1]$.
    We claim that
	\begin{equation}\label{eq:nest_j}
		J_p \subsetneq J_q \quad \text{for } p<q \leq n.
	\end{equation}
	Suppose that $J_{t-1} = J_t$ for some~$t < n$.
    Then $I_t \subset \cup_{s < t}I_s$.
    Pick all maximal elements among $I_1, \dots, I_{t-1}$.
    There are at least two such elements since $I_t \not \subset I_s$ for $s < t$.
    Maximal elements are pairwise disjoint by the condition~\ref{a}.
    Now their union is not in~$\cB$ by the condition~\ref{b}, which contradicts the assumption $I_t \in \cB$.
    The case when $t=n$ is also impossible since $J_{n-1} \notin \cB$ by the condition~\ref{b} but $[n+1] \in \cB$.
	
	Note from the condition~\ref{a} that the equations~\eqref{eqn:edge_in_hyperplane} and~\eqref{eqn:edge} can be written as
	\begin{equation}\label{eq:zerosum}
		\sum_{i \in J_t} a_i = 0 \quad \text{for } t=1, \ldots, n.
	\end{equation}
	It follows from~\eqref{eq:nest_j} that there exists $0 \leq t_0 < n$ such that $|J_{t_0+1}| = |J_{t_0}| + 2$ and $|J_{t+1}| = |J_t| + 1$ for $t \neq t_0$.
    Let $j$ and $k$ be the elements of $J_{t_0+1} \setminus J_{t_0}$.
    Now the equation~\eqref{eq:zerosum} shows that $a_j + a_k = 0$ and $a_i = 0$ for $i \neq j, k$.
    Therefore, $\eta = e_j - e_k$ up to sign.
\end{proof}

\subsection{The case when $G$ is connected}
Now we focus on the case when the building set is given by a simple graph~$G$ with the vertex set~$[n+1]$. 
We assume that $G$ is connected for the moment and return to the general case in the next subsection.
Recall from~\eqref{eq:graph_building_set} that the facets of the graph associahedron~$P_{\cB(G)}$ are labeled by nonempty proper subsets~$I$ of~$[n+1]$ such that the induced subgraph~$G|_I$ is connected. 
Recall also that $k_i = k_i(G)$ denotes the number of connected induced subgraphs of~$G$ containing the vertex~$i$. To simplify the argument, by relabeling the vertices if necessary, we will assume that $k_{i}$ is minimal when $i=n+1$.

The following is a restatement of our main theorem with the assumption that $G$ is connected.
\begin{theorem}\label{thm:restate_main}
    Let $G$ be a connected simple graph with the vertex set $[n+1]$ such that $k_{i}$ is minimal when $i=n+1$.
    Then
    $$
        w_G (M_G, \omega_G) = k_{n+1} - 1.
    $$
\end{theorem}

The following two graph theoretic lemmas will be used in the proof of Theorem~\ref{thm:restate_main}.
\begin{lemma} \label{lem:parallel_facets}
    The subgraph $G|_{[n]}$ induced by $[n] = [n+1]\setminus \{n+1\}$ is connected.
\end{lemma}
\begin{proof}
	It holds when $n =1$. Suppose $G|_{[n]}$ is not connected with $n >1$. Then $G|_{[n]}$ has at least two connected components, say $G_1$ and $G_2$. By relabeling the vertices if necessary, we may assume that $1 \in G_1$, $2 \in G_2$ and $n+1$ is adjacent to both~$1$ and~$2$. Let $\ell_i$ (respectively, $m_i$) denote the number of connected induced subgraphs containing~$i$ and also containing $n+1$ (respectively, but not containing $n+1$). Then $k_i = \ell_i + m_i$. We assume $m_1 \leq m_2$ without loss of generality.
	
	A connected induced subgraph of~$G$ containing~$n+1$ satisfies one of the following three mutually exclusive conditions:
	\begin{itemize}
		\item it contains~$1$,
		\item it contains~$2$ but does not contain~$1$,
		\item it does not contain either~$1$ or~$2$.
	\end{itemize}
	Note that there are at least $m_2$ graphs satisfying the second condition. Hence, we have $k_{n+1} \geq \ell_1 + m_2 + 1$. Since $k_1 = \ell_1 + m_1$ and $m_1 \leq m_2$, we conclude that $k_{n+1} > k_1$. This contradicts the minimality of~$k_{n+1}$.
\end{proof}

\begin{lemma}\label{lem:inequality}
	Let $G$ be as above and $\cB:= \cB(G)$ the building set obtained from~$G$. Then we have
	\begin{equation}\label{eq:inequality}
		n \cdot k_i(G) \geq |\cB| -1
	\end{equation}
	for all $i=1, \dots, n+1$ and $n \geq 1$.
\end{lemma}
\begin{proof}
	We prove the lemma by induction on~$n$. The case when $n=1$ is trivial. Now suppose that \eqref{eq:inequality} holds for all connected simple graphs with at most $n$ vertices.
    Set $G' = G|_{[n]}$.
    By the induction hypothesis, the following inequality
	\[
		(n-1)k_i(G') \geq |\cB|_{[n]}| - 1
	\]
	holds for all $i=1, \dots, n$. Let $m$ be a vertex of~$G$ adjacent to the vertex~$n+1$.
    Then
	\[
		k_i(G) \geq k_{n+1}(G) > k_m(G')  \geq \frac{|\cB|_{[n]}| - 1}{n-1}.
	\]
	This implies
	\[
		(n-1) k_i(G) > |\cB|_{[n]}| - 1 = (|\cB| - k_{n+1}(G)) - 1 \geq (|\cB| - k_i(G)) - 1,
	\]
	which proves \eqref{eq:inequality}.
\end{proof}

In order to prove Theorem~\ref{thm:restate_main}, we identify the hyperplane $H_{\cB(G)}$ with $\R^n$ by the projection
$(x_1, \ldots, x_n, x_{n+1}) \mapsto (x_1, \ldots, x_n)$.
This projection sends the lattice $H_{\cB(G)} \cap \Z^{n+1}$ isomorphically onto~$\Z^n$. Let $P_G \subset \R^n$ be the image of $P_{\cB(G)}$ under this identification. Set
\begin{equation}\label{eq:a}
	a:= \frac{|\cB(G)| - k_{n+1} - 1}{n-1}.
\end{equation}
Let $L_i$ be the line segment whose $i$th coordinate varies from~$1$ to $k_{n+1}$ while all the other coordinates remain to be~$a$.

\begin{lemma} \label{lem:lower_bound}
The line segment~$L_i$ is contained in~$P_G$ for all $i=1, \dots, n$. Moreover, they intersect at one point $(a, \dots, a) \in P_G$.
\end{lemma}
\begin{proof}
    Set $\cB := \cB(G)$ for brevity.
	For $(x_1, \ldots, x_{n}) \in L_i$ and $I (\neq [n+1]) \in \cB$, we will show $\sum_{i \in I} x_i \geq | \cB|_I|$ if $n+1 \not\in I$, and $\sum_{i \not\in I} x_i \leq |\cB| - |\cB|_I|$ if $n+1 \in I$.
    It is enough to check the following inequalities:
\begin{alignat*}{2}
    1 + (|I| - 1)a &\geq | \cB|_I|  \quad\quad && \text{ if $n+1 \not\in I$, and} \\
    k_{n+1} + (n - |I|)a &\leq |\cB| - |\cB|_I|, \quad\quad && \text{ if $n+1 \in I$.}
\end{alignat*}
    This is obvious when $|I| = 1$ and both inequalities are equivalent to the inequality
	\begin{equation}\label{eq:increasing}
		\frac{|\cB| - k_{n+1} - 1}{n-1} \geq \frac{|\cB|_I| - 1}{|I| - 1}
	\end{equation}
	when $|I|>1$.

    We define a function $f \colon \cB \setminus \{ \{1\}, \ldots, \{n+1\}\} \to \Q$ by $f(I) = \frac{|\cB|_I| -1}{|I|-1}$.
    Suppose there exists $J \in \cB$ such that $J = I \cup \{m\}$ for some $m \in [n+1] \setminus I$.
    Then Lemma~\ref{lem:inequality} applied for~$G|_J$ and the equality $|\cB|_J| = |\cB|_I| + k_m(G|_J)$ show that
	\[
		k_m(G|_J) \geq \frac{|\cB|_I| - 1}{|I| - 1}.
	\]
	This implies
	\begin{equation}\label{eq:nested_increasing}
		f(J) = \frac{|\cB|_J|-1}{|J| - 1} = \frac{|\cB|_I| + k_m(G|_J) - 1}{|I|} \geq \frac{|\cB|_I| - 1}{|I| - 1} = f(I).
	\end{equation}
	Therefore, $f$ is an increasing function on the size of the nested index sets.

    Since $\cB$ is a building set constructed from a connected simple graph, for any $I \neq [n+1]$ in~$\cB$, there is $J \in \cB$ such that $I \subset J$ and $[n+1] \setminus J = \{j\}$ for some~$j$.
    By \eqref{eq:nested_increasing} and the minimality of~$k_{n+1}$, we have
    $$
        f(I) \leq f(J) = \frac{ |\cB|_J| -1 }{|J|-1} =\frac{ |\cB| - k_j -1 }{n-1} \leq \frac{ |\cB| - k_{n+1} -1 }{n-1},
    $$ which proves the inequality~\eqref{eq:increasing}, hence the first statement.
	
	It remains to show that $1 \leq a \leq k_{n+1}$. The first inequality is clear from the definition and the second inequality follows again from Lemma~\ref{lem:inequality}.
\end{proof}

\begin{proof}[Proof of Theorem~\ref{thm:restate_main}]
The line segments~$L_i$ from Lemma~\ref{lem:lower_bound} have affine length $k_{n+1} - 1$ for all $i=1, \dots, n$. They intersect at the point $(a, a, \dots, a)$ and the primitive vectors parallel to $L_i$'s are standard basis vectors.
Therefore, by Proposition~\ref{prop:lower_bound}, we have $w_G(M_G, \omega_G) \geq k_{n+1}-1$.

On the other hand, by Lemma~\ref{lem:parallel_facets}, both $\{n+1\}$ and $[n]$ are in $\cB:=\cB(G)$. These correspond to the hyperplanes $x_{n+1} \geq 1$ and $x_{n+1} \leq |\cB| - |\cB|_{[n]}|$, respectively.
Since any element in~$\cB$ either contains the vertex~$n+1$ or is contained in the set~$[n]$, the latter hyperplane is $x_{n+1} \leq k_{n+1}$.
Recall that the Delzant polytope~$P_G$ is obtained from~$P_{\cB(G)}$ by forgetting the last coordinate. Hence, we have two hyperplanes $\{ x \in \R^n \mid \sum_{i=1}^n x_i \leq |\cB| - 1 \}$ and $\{ x \in \R^n \mid \sum_{i=1}^n x_i \geq |\cB| - k_{n+1} \}$ supporting~$P_G$. By Lemma~\ref{lem:edge_property}, any primitive vector parallel to an edge of $P_G$ must be of the form $\pm e_j$, or $e_j - e_k$ for some $j \neq k$. We apply Proposition~\ref{prop:key} for $u = (1, \dots, 1)$ and obtain $w_G(M_G, \omega_G) \leq k_{n+1}-1$ as desired.
\end{proof}

\begin{proof}[Proof of Corollary~\ref{cor:embedding}]
By the proof of Theorem~\ref{thm:restate_main}, the Gromov width of~$M_G$ is given by $K_{\max} - K_{\min}$ for some moment map $K \colon M_G \to \mathbb{R}$ given by a semifree $S^1$-action. The inequality $w_G(M_G \times \mathbb{R}^{2m}) \leq  w_G(M_G)$ follows from Corollary~\ref{cor:stabilization}. The obvious embedding $B^{2n+2m}(r) \hookrightarrow B^{2n}(r) \times \mathbb{R}^{2m}$ shows that this is actually an equality. Similarly, $w_G(M_H \times \mathbb{R}^{2k+2m}) = w_G(M_H)$. Now the corollary follows from Corollary~\ref{cor:subgraph}.
\end{proof}

\subsection{General case}
Let $G_1, \dots, G_s$ be the connected components of the graph~$G$. By the definition of~$P_{\cB(G)}$ as the Minkowski sum~\eqref{eq:Minkowski_sum}, we have
\[
	P_{\cB(G)} = \prod_{j=1}^s P_{\cB(G_j)} \subset \prod_{j=1}^s \mathbb{R}^{|G_j|} = \mathbb{R}^{n+1}.
\]
Note that the Minkowski sum with a point is just a translation and does not affect the shape of the polytope. So we will ignore the components consisting of one point and assume that $n_j:= |G_j| -1 \geq 1$ for all $j=1, \dots, s$. Now $(M_G, \omega_G)$ is the product of the symplectic manifolds $(M_{G_j}, \omega_{G_j})$. The vertices of~$G$ will be labeled using two indices~$i$ and~$j$, and we let $k_i^j$ denote the number of connected induced subgraphs containing such vertex. As in the connected case, we assume that the minimum among $k_1^j, \dots, k_{n_j+1}^j$ is attained when $i=n_j+1$ for each $j=1, \dots, s$.

\begin{proof}[Proof of Theorem~\ref{thm:main}]
	By Theorem~\ref{thm:restate_main}, the Gromov width of~$M_{G_j}$ is~$k_{n_j+1}^j$ for $j=1, \dots, s$. If $k_{n_m+1}^m$ is minimal among them, consider the $S^1$-action on~$M_{G_m}$ whose moment map~$H$ satisfies $H_{\max} - H_{\min} = k_{n_m+1}^m - 1$. As in the proof of Corollary~\ref{cor:stabilization}, we extend this action to~$M_G$ by acting trivially on the other factors. Then $w_G(M_G, \omega_G) \leq k_{n_m+1}^m - 1$ follows from Lemma~\ref{lem:key}.
	
	The other direction follows from the inclusion
	\[
		B^{2(n-s+1)}(r) \hookrightarrow \prod_{j=1}^s B^{2n_j}(r) \hookrightarrow \prod_{j=1}^s M_{G_j},
	\]
	where $\pi r^2 = k_{n_m+1}^m -1$.
\end{proof}

\begin{remark}
	When $G=K_{n+1}$, the corresponding polytope $P_{\cB(K_{n+1})}$ is a permutohedron whose vertices are obtained by all permutations of the coordinates of the point $(2^0, 2^1, \ldots, 2^n) \in \R^{n+1}$. By Theorem~\ref{thm:main}, we have $w_G(M_{K_{n+1}}, \omega_{K_{n+1}}) = 2^n - 2^0$.

	More generally, we can think of a permutohedron $P$ whose vertices are obtained by all permutations of the coordinates of the point
    $$
        (c_1, \ldots, c_{n+1}) \in~\R^{n+1} \quad \text{ with } c_1 < c_2 < \ldots < c_{n+1}.$$
    As before, we consider $P$ as a Delzant polytope in~$\R^n$ by forgetting the last coordinate. The corresponding symplectic toric manifold $(M, \omega)$ has the Gromov width given by
	\begin{equation}
		w_G (M, \omega) = c_{n+1} - c_1.
	\end{equation}
	The upper bound is obtained from two supporting hyperplanes $x_1 \geq c_1$ and $x_1 \leq c_{n+1}$. For the lower bound, we may take $L_i$ to be the line segment whose $i$th coordinate varies from $c_1$ to $c_{n+1}$ and all the other coordinates are $\sum_{i=2}^n c_i / (n-1)$.
\end{remark}

\begin{remark}
    It would be interesting to compute the Gromov width of the symplectic toric manifolds corresponding to general nestohedra.
    It is reasonable to define $k_i$ to be the number of elements in the building set containing $i$ and ask whether Therem~\ref{thm:main} still holds.
    However, Theorem~\ref{thm:main} does not hold for general nestohedra.
    For instance, let
	\[
		\cB:= \{ \{1\}, \{2\}, \{3\}, \{4\}, \{1,2\}, \{3,4\}, \{1,2,3,4\}\}
	\]
	be a building set on~$\{1,2,3,4\}$. Then the defining inequalities are given by
	\[
		x_i \geq 1 \text{ (for } 1 \leq i \leq 3), \quad x_1 + x_2 + x_3 \leq 6, \quad 3 \leq x_1 + x_2 \leq 4.
	\]
	Note that $x_1 + x_2 = 3$ and $x_1 + x_2 = 4$ are parallel hyperplanes supporting~$P_{\cB}$. The Gromov width is at most~$1$ by Proposition~\ref{prop:key} and Lemma~\ref{lem:edge_property}, while the minimum of $k_i - 1$ is~$2$.
\end{remark}

\section*{acknowledgement}
The main problem of this paper was motivated by the online seminar series organized by Prof. Dong Youp Suh and Dr. Jongbaek Song in 2020. The authors thank the organizers.

\bigskip
\bibliographystyle{amsplain}
%\bibliography{reference2020}

\providecommand{\bysame}{\leavevmode\hbox to3em{\hrulefill}\thinspace}
\providecommand{\MR}{\relax\ifhmode\unskip\space\fi MR }
\providecommand{\MRhref}[2]{%
  \href{http://www.ams.org/mathscinet-getitem?mr=#1}{#2}
}
\providecommand{\href}[2]{#2}

\end{document}